\documentclass{amsart}

\usepackage{amssymb,amsmath,amsthm,latexsym}

\theoremstyle{plain}
\newtheorem{lemma}{Lemma}[section]
\newtheorem{proposition}[lemma]{Proposition}
\newtheorem{theorem}[lemma]{Theorem}
\newtheorem{corollary}[lemma]{Corollary}

\theoremstyle{definition}
\newtheorem{definition}[lemma]{Definition}
\newtheorem{notation}[lemma]{Notation}
\newtheorem{remark}[lemma]{Remark}

\begin{document}

\title[Enveloping algebras of solvable Malcev algebras]
{Enveloping algebras of solvable Malcev algebras of dimension five}

\author[Bremner]{Murray R. Bremner}

\address{Department of Mathematics and Statistics,
University of Saskatchewan, Canada}

\email{bremner@math.usask.ca}

\author[Tvalavadze]{Marina V. Tvalavadze}

\address{Department of Mathematics and Statistics,
University of Saskatchewan, Canada}

\email{tvalavadze@math.usask.ca}

\date{\textit{\today}}

\begin{abstract}
We study the universal enveloping algebras of the one-parameter family of
solvable 5-dimensional non-Lie Malcev algebras.  We explicitly determine the
universal nonassociative enveloping algebras (in the sense of P\'erez-Izquierdo
and Shestakov) and the centers of the universal enveloping algebras.  We also
determine the universal alternative enveloping algebras.
\end{abstract}



\maketitle

\section{Introduction}

In this paper we study the 5-dimensional solvable (non-nilpotent non-Lie)
Malcev algebras and their universal enveloping algebras. Over a field of
characteristic 0, the 5-dimensional Malcev algebras were classified by Kuzmin
\cite{Kuzmin}: there is one nilpotent algebra, solvable algebras of five
different types, and one non-solvable algebra. Except for a finite number of
special cases, the solvable algebras belong to a family whose structure
constants involve a non-zero parameter $\gamma$. Kuzmin omitted the details of
the classification in the solvable case, but Gavrilov \cite{Gavrilov} has
recently recovered these results using Malcev cocycles.

In 2004, P\'erez-Izquierdo and Shestakov \cite{PerezIzquierdoShestakov}
extended the Poincar\'e-Birkhoff-Witt (PBW) theorem from Lie algebras to Malcev
algebras. For any Malcev algebra $M$ over a field of characteristic 0 or $p
> 3$, they constructed a universal nonassociative enveloping algebra
$U(M)$ which shares many properties of the universal associative enveloping
algebras of Lie algebras. In general $U(M)$ is not alternative, and so it is
interesting to determine its alternator ideal $I(M)$ and its maximal
alternative quotient $A(M) = U(M)/I(M)$, which is the universal alternative
enveloping algebra of $M$. This produces new examples of infinite dimensional
alternative algebras. The details have been worked out by Bremner, Hentzel,
Peresi and Usefi \cite{BHPU,BremnerUsefi} for the 4-dimensional solvable
algebra and the 5-dimensional nilpotent algebra.  See also the survey article
\cite{BHPTU}.

The goal of this paper is to compute explicit structure constants for
$U(\mathbb{M})$ and $A(\mathbb{M})$ where $\mathbb{M} = \mathbb{M}_\gamma$
belongs to the one-parameter family of 5-dimensional solvable Malcev algebras.
We also determine the center of $U(\mathbb{M})$ which is non-trivial if and
only if the parameter $\gamma$ is rational.

We recall the structure constants of $\mathbb{M}=\text{span}\{e_0, e_1, e_2,
e_3, e_4\}$ from \cite{Kuzmin}:
  \[
  e_1 e_2 = e_3, \quad
  e_0 e_1 = e_1, \quad
  e_0 e_2 = e_2, \quad
  e_0 e_3 = - e_3, \quad
  e_0 e_4 = \gamma' e_4 \; (\gamma' \ne 0).
  \]
Thus $e_1, e_2, e_3$ span a 3-dimensional nilpotent Lie algebra, and $e_1, e_2,
e_3, e_4$ span the direct sum of this Lie algebra with a 1-dimensional abelian
Lie algebra. The basis element $e_0$ acts diagonally on this 4-dimensional
nilpotent Lie algebra, producing a 5-dimensional solvable Malcev algebra; the
parameter enters only into the action of $e_0$ on $e_4$. We change notation,
replacing $e_0$ by $-a$, $e_1$ by $b$, $e_2$ by $c$, $e_3$ by $2 d$, $e_4$ by
$e$ (and $\gamma'$ by $-\gamma$), and obtain the following structure constants
for $\mathbb{M}$:
  \begin{equation}
  \label{oneparameterfamily}
  [ b, c ] = 2 d, \quad
  [ a, b ] = - b, \quad
  [ a, c ] = - c, \quad
  [ a, d ] = d, \quad
  [ a, e ] = \gamma e \; (\gamma \ne 0).
  \end{equation}
The span of $a, b, c, d$ is the 4-dimensional solvable (non-Lie) Malcev
algebra; see \cite{BHPU}.


\section{Preliminaries}

\begin{definition}
The \textbf{generalized alternative nucleus} of a nonassociative algebra $A$
over a field $F$ is the subspace
\[
 N_{\mathrm{alt}}(A)
  =
  \big\{ \,
  s \in A
  \mid
  (s,x,y) = -(x,s,y) = (x,y,s), \, \forall \, x, y \in A
  \, \big\}.
  \]
\end{definition}

In general $N_{\mathrm{alt}}(A)$ is not a subalgebra of $A$, but it is a
subalgebra of $A^-$ (it is closed under the commutator) and is in fact a Malcev
algebra.

\begin{theorem}
\emph{(P\'erez-Izquierdo and Shestakov \cite{PerezIzquierdoShestakov})} For
every Malcev algebra $M$ over a field $F$ of characteristic $\ne 2, 3$ there
exists a nonassociative algebra $U(M)$ and an injective algebra morphism
$\iota\colon M \to U(M)^-$ such that $\iota(M) \subseteq
N_{\mathrm{alt}}(U(M))$; furthermore, $U(M)$ is a universal object with respect
to such morphisms.
\end{theorem}

The algebra $U(M)$ is constructed as follows. Let $F(M)$ be the unital free
nonassociative algebra on a basis of $M$. Let $R(M)$ be the ideal of $F(M)$
generated by the following elements for all $s, t \in M$ and all $x, y \in
F(M)$:
  \[
  st - ts - [s,t],
  \qquad
  (s,x,y) + (x,s,y),
  \qquad
  (x,s,y) + (x,y,s).
  \]
Define $U(M) = F(M)/R(M)$ with the natural mapping
  \[
  \iota\colon M \to N_\mathrm{alt}(U(M)) \subseteq U(M),
  \qquad
  s \mapsto \iota(s) = \overline{s} = s + R(M).
  \]
Since $\iota$ is injective, we may identify $M$ with $\iota(M) \subseteq U(M)$.
We fix a basis $B = \{ a_i \,|\, i \in \mathcal{I} \}$ of $M$ and a total order
$<$ on $\mathcal{I}$, and define
  \[
  \Omega =
  \{ \,
  (i_1,\hdots,i_n)
  \mid
  n \ge 0, \,
  i_1 \le \cdots \le i_n \,
  \}.
  \]
For $n = 0$ we have $\overline{a}_\emptyset = 1 \in U(M)$, and for $n \ge 1$
the $n$-tuple $(i_1,\hdots,i_n) \in \Omega$ defines a left-tapped monomial
  \[
  \overline{a}_I =
  \overline{a}_{i_1} (
  \overline{a}_{i_2} ( \cdots
  ( \overline{a}_{i_{n-1}} \overline{a}_{i_n} )
  \cdots )),
  \qquad
  |\overline{a}_I| = n.
  \]
In \cite{PerezIzquierdoShestakov} it is shown that the set of all
$\overline{a}_I$ for $I \in \Omega$ is a basis of $U(M)$.  It follows that
there is a linear isomorphism $\phi\colon U(M) \to P(M)$ which is the identity
on $M$, where $P(M)$ is the polynomial algebra on $M$. Since $M \subseteq
N_\mathrm{alt}(U(M))$, for any $s, t \in M$ and $x \in U(M)$ we have
  \[
  (s,t,x)
  =
  \tfrac16
  [[x,s],t]
  -
  \tfrac16
  [[x,t],s]
  -
  \tfrac16
  [x,[s,t]].
  \]
This equation implies the following lemma, which is implicit in
\cite{PerezIzquierdoShestakov}.

\begin{lemma} \label{shestakovlemma}
Let $x$ be a basis monomial of $U(M)$ with $|x| \ge 2$ and write $x = ty$ with
$t \in M$.  Then for any $s \in M$ we have
  \allowdisplaybreaks
  \begin{align}
  [x,s]
  &=
  [t,s]y
  +
  t[y,s]
  +
  \tfrac12
  [[y,s],t]
  -
  \tfrac12
  [[y,t],s]
  -
  \tfrac12
  [y,[s,t]],
  \label{rightbracketformula}
  \\
  sx
  &=
  t(sy)
  +
  [s,t]y
  -
  \tfrac13
  [[y,s],t]
  +
  \tfrac13
  [[y,t],s]
  +
  \tfrac13
  [y,[s,t]].
  \label{leftproductformula}
  \end{align}
Let $y$ be a basis monomial of $U(M)$ with $|y| \ge 2$ and write $y = sx$ with
$s \in M$. Then for any basis monomial $z$ of $U(M)$ we have
  \begin{equation}
  yz = (sx)z = 2 s(xz) - x(sz) - x[z,s] + [xz,s].
  \end{equation}
\end{lemma}

\begin{definition}
In the nonassociative algebra $A$, we write $L_s$ and $R_s$ for the operators
of left and right multiplication by $s$, and set $\mathrm{ad}_s = R_s - L_s$.
We define
 \[
  D_{s,t} = [L_s, L_t] + [L_s, R_t] + [R_s, R_t].
 \]
We note that $D_{t,s} = - D_{s,t}$ and $D_{s,s} = 0$.
\end{definition}

In Table \ref{derivationtable} we record the values of the derivations
$D_{s,t}$ on the one-parameter family $\mathbb{M}_\gamma$ of 5-dimensional
solvable Malcev algebras (dot indicates zero).

The following lemma is proved by Morandi, P\'erez-Izquierdo and Pumpl\"un
\cite[Lemma 4.2]{Morandi}; note that we have changed the sign of the
$\mathrm{ad}$-operator.

\begin{lemma} \label{morandilemma}
If $A$ is a nonassociative algebra with $s, t \in N_{\mathrm{alt}}(A)$ and $x
\in A$ then
  \allowdisplaybreaks
  \begin{alignat*}{3}
  &
  L_{sx} = L_s L_x + [R_s, L_x],
  &\quad
  &
  L_{xs} = L_x L_s + [L_x, R_s],
  &\quad
  &
  [L_s, L_t] = L_{[s,t]} - 2 [R_s, L_t],
  \\
  &
  R_{sx} = R_x R_s + [R_x, L_s],
  &\quad
  &
  R_{xs} = R_s R_x + [ L_s, R_x],
  &\quad
  &
  [R_s, R_t] = - R_{[s,t]} - 2 [L_s, R_t],
  \\
  &
  [L_s, R_t] = [R_s, L_t].
  \end{alignat*}
The operator $D_{s,t}$ is a derivation, and we have
  \[
  D_{s,t} = -\mathrm{ad}_{[s,t]} - 3 [L_s, R_t],
  \qquad
  2 D_{s,t} = - \mathrm{ad}_{[s,t]} + [\mathrm{ad}_s, \mathrm{ad}_t].
  \]
\end{lemma}

Using $D_{s,t}$ we can rewrite the equations \eqref{rightbracketformula} and
\eqref{leftproductformula} as follows:
  \allowdisplaybreaks
  \begin{align}
  [ ty, s ] &= [t,s] y + t [y,s] - D_{s,t}(y) - [y,[s,t]],
  \label{rightbracket}
  \\
  s(ty) &= t(sy) + [s,t] y + \tfrac23 D_{s,t}(y) + \tfrac23 [y,[s,t]].
  \label{leftproduct}
  \end{align}

  \begin{table}
  \begin{center}
  \[
  \begin{array}{l|rrrrrrrrrr}
    & D_{a,b} & D_{a,c} & D_{a,d} & D_{a,e} & D_{b,c}
    & D_{b,d} & D_{b,e} & D_{c,d} & D_{c,e} & D_{d,e} \\ \hline
    &&&&&&&&&& \\[-8pt]
  a & -b & -c & -d & -\gamma^2 e & d & . & . & . & . & . \\
  b &  . & -d &  . & . & . & . & . & . & . & . \\
  c &  d &  . &  . & . & . & . & . & . & . & . \\
  d &  . &  . &  . & . & . & . & . & . & . & . \\
  e &  . &  . &  . & . & . & . & . & . & . & .
  \end{array}
  \]
  \end{center}
  \caption{Derivations of the one-parameter family of Malcev algebras}
  \label{derivationtable}
  \end{table}


\section{Left multiplications on solvable Malcev algebras}

We recall the classification of 5-dimensional solvable Malcev algebras from
\cite{Kuzmin}. We omit the first case (the direct sum of the 4-dimensional
solvable algebra and a 1-dimensional abelian algebra) and consider only the
remaining five cases:
  \allowdisplaybreaks
  \begin{alignat}{5}
  [b,c] &= 2d, &\;\;
  [a,b] &=  b, &\;\;
  [a,c] &=  c, &\;\;
  [a,d] &= 2d - \tfrac12 e, &\;\;
  [a,e] &= -e;
  \label{kuzmin8}
  \\
  [b,c] &= 2d, &\;\;
  [a,b] &=  b, &\;\;
  [a,c] &=  c, &\;\;
  [a,d] &= -d, &\;\;
  [a,e] &= -b-2e;
  \label{kuzmin9}
  \\
  [b,c] &= 2d, &\;\;
  [a,b] &=  b, &\;\;
  [a,c] &=  c, &\;\;
  [a,d] &= -d, &\;\;
  [a,e] &= 2d-e;
  \label{kuzmin10}
  \\
  [b,c] &= 2d, &\;\;
  [a,b] &=  b+e, &\;\;
  [a,c] &=  c, &\;\;
  [a,d] &= -d, &\;\;
  [a,e] &=  e;
  \label{kuzmin11}
  \\
  [b,c] &= 2d, &\;\;
  [a,b] &= -b, &\;\;
  [a,c] &= -c, &\;\;
  [a,d] &=  d, &\;\;
  [a,e] &= \gamma e \;\; (\gamma \ne 0).
  \label{kuzmin12}
  \end{alignat}
(We have made a slight change of basis from \cite{Kuzmin} so that $[b,c] = 2d$
in every case.) Equation \eqref{kuzmin12} coincides with equation
\eqref{oneparameterfamily}; this is the one-parameter family on which we focus
in this paper. But for the rest of this section we work more generally and let
$M =\text{span}\{a,b,c,d,e\}$ be one of the algebras in equations
\eqref{kuzmin8}, \eqref{kuzmin10}, \eqref{kuzmin11}, \eqref{kuzmin12}. Then $L
= \text{span}\{b,c,d,e\}$ is a 4-dimensional nilpotent Lie algebra with $[b,c]
= 2d$ and other products zero; furthermore, $[a,L] = L$.

\begin{lemma} \label{zeroderivations}
If $M$ is one of the algebras \eqref{kuzmin8}, \eqref{kuzmin10},
\eqref{kuzmin11}, \eqref{kuzmin12} then in $U(M)$ we have
  \[
  D_{b,d} = D_{b,e} = D_{c,d} = D_{c,e} = D_{d,e} = 0.
  \]
\end{lemma}

\begin{proof}
We have $D_{s,t}= -\frac12 \mathrm{ad}_{[s,t]} + \frac12 [\mathrm{ad}_s,
\mathrm{ad}_t]$ where $s, t \in M$. If $D_{s,t}$ is one of the above
derivations then $[s,t] = 0$ and so $\mathrm{ad}_{[s,t]} = 0$; hence $D_{s,t} =
\frac12 [\mathrm{ad}_s, \mathrm{ad}_t ]$. From the multiplication table of $M$
we have $\mathrm{ad}_d(M) \subseteq \mathrm{span} \{ d, e \} $ and
$\mathrm{ad}_e (M) \subseteq \mathrm{span} \{ d,e \}$. Hence $\mathrm{ad}_s (
\mathrm{ad}_d (M) ) = 0$ and $\mathrm{ad}_s ( \mathrm{ad}_e (M) ) = 0$ for $s
\in \{b,c,d,e\}$ (in fact $s \in L$). On the other hand, $\mathrm{ad}_d |_L =0$
and $\mathrm{ad}_e |_L =0$. It follows that $D_{s,t}(M) = \frac12 \mathrm{ad}_s
\mathrm{ad}_t(M) - \frac12 \mathrm{ad}_t\mathrm{ad}_s (M) =0$ for $t \in \{ d,
e \}$ and $ s \in \{b,c,d,e\}$. Since $D_{s,t}$ is a derivation on $U(M)$, we
have $D_{s,t}(U(M))=0$.
\end{proof}

\begin{lemma}\label{powerofL}
In $U(M)$ we have $L_{s^m} = L_s^m$ for all $s \in M$.
\end{lemma}

\begin{proof}
Lemma \ref{morandilemma} implies that $L_{s^2} = L_s L_s + [R_s, L_s]$ for $s
\in N_{\mathrm{alt}}(U(M))$. Since $s \in M \subseteq N_{\mathrm{alt}}(U(M))$
we have
  \[
  [ R_s, L_s ] (x) = R_s L_s (x) - L_s R_s (x) = (sx)s - s(xs) = (s,x,s) = 0.
  \]
Hence $L_{s^2} = L_s^2$. We now prove that $L_{s^m} = L_s^m$ for $m \ge 3$ by
induction on $m$. By the inductive hypothesis and Lemma \ref{morandilemma} we
have
  \allowdisplaybreaks
  \begin{align*}
  L_{s^{m+1}}
  =
  L_{s s^m}
  =
  L_s L_{s^m} + [ R_s , L_{s^m} ]
  =
  L_s L_s^m + [ R_s, L_s^m ]
  =
  L_s^{m+1},
  \end{align*}
since $R_s$ commutes with $L_s$ (and hence it commutes with $L_s^m$).
\end{proof}

\begin{lemma} \label{cde-operator}
In $U(M)$ the operators $L_c$, $L_d$, $L_e$ are pairwise commutative, and
  \[
  L_{c^k d^l e^m} = L^k_c L^l_d L^m_e.
  \]
\end{lemma}

\begin{proof}
Lemma \ref{morandilemma} implies that for any $x, y \in N_{\mathrm{alt}}(U(M))$
we have
  \[
  [L_x, L_y] = L_{[x,y]} - 2 [R_x, L_y],
  \quad
  D_{x,y}= -\mathrm{ad}_{[x,y]} - 3 [L_x, R_y],
  \quad
  [L_x, R_y] = [R_x, L_y].
  \]
If $[x,y] = 0$ then $L_{[x,y]} = 0$ and $\mathrm{ad}_{[x,y]} = 0$, and hence
  \[
  [ L_x, L_y ] = \tfrac23 D_{x,y}.
  \]
Therefore by Lemma \ref{zeroderivations} we have
  \[
  [ L_c, L_d ] = [ L_c, L_e ] = [ L_d, L_e ] = 0.
  \]
By Lemma \ref{powerofL} we have
  \[
  L_{c^k} = L^k_c,
  \quad
  L_{d^l} = L^l_d,
  \quad
  L_{e^m} = L^m_e.
  \]
If $[x,y] = 0$ then $D_{x,y} = -3 [ L_x, R_y ] = -3 [ R_x, L_y ]$, and so Lemma
\ref{zeroderivations} implies
  \[
  [ L_c, R_d ] = [ R_c, L_d ] = 0,
  \quad
  [ L_c, R_e ] = [ R_c, L_e ] = 0,
  \quad
  [ L_d, R_e ] = [ R_d, L_e ] = 0.
  \]
We show by induction on $l$ that
  \[
  L_{ d^l e^m } = L_d^l L_e^m.
  \]
The basis $l = 0$ is Lemma \ref{powerofL}. If $L_{d^{l-1} e^m} = L^{l-1}_d
L^m_e$ for some $l \ge 1$ then
  \allowdisplaybreaks
  \begin{align*}
  L_{ d^l e^m }
  &=
  L_{ d ( d^{l-1}e^m ) }
  =
  L_d  L_{ d^{l-1} e^m } + [ R_d, L_{d^{l-1} e^m} ]
  =
  L_d^l L_e^m + [ R_d, L_d^{l-1} L_e^m ]
  \\
  &=
  L_d^l L_e^m + [ R_d, L_d^{l-1} ] L_e^m + L_d^{l-1} [ R_d, L_e^m ]
  =
  L_d^l L_e^m,
  \end{align*}
since $R_d$ commutes with $L_e^m$ and $L_d^{l-1}$. We show by induction on $k$
that
  \[
  L_{c^k d^l e^m}= L^k_c L^l_d L^m_e.
  \]
The basis $k = 0$ is the previous formula. If $ L_{c^{k-1} d^l e^m}= L^{k-1}_c
L^l_d L^m_e$ for $k \ge 1$ then
  \allowdisplaybreaks
  \begin{align*}
  L_{c^k d^l e^m }
  &=
  L_{ c ( c^{k-1} d^le^m ) }
  =
  L_c  L_{c^{k-1} d^l e^m } + [ R_c, L_{c^{k-1} d^l e^m} ]
  \\
  &=
  L_c^k L_d^l L_e^m + [ R_c, L_c^{k-1} L_d^l L_e^m ]
  =
  L_c^k L_d^l L_e^m + [ R_c, L_c^{k-1} ] L_d^l L_e^m
  =
  L_c^k L_d^l L_e^m
  \end{align*} since $R_c$ commutes with $L^{k-1}_c$, $L^l_d$ and $L^m_e$.
The proof is complete.
\end{proof}

\begin{notation}
We set $[ \underbrace{ X, \dots, X }_q, Y ] = [ \underbrace{ X, \hdots, [ X, [
X}_q, Y ] ] \cdots ]$; if $q = 0$ we get $Y$.
\end{notation}

\begin{lemma} \label{bigmorandilemma}
If $s\in N_{\mathrm{alt}} (U(M))$ then
  \[
  L_{ s^k x }
  =
  \sum^k_{q=0}
  \binom{k}{q}
  L^{k-q}_s
  [ \underbrace{R_s,\dots,R_s}_q, L_x ].
  \]
\end{lemma}

\begin{proof}
By induction on $k$; the basis $k = 0$ is trivial, and $k = 1$ is the first
equation in Lemma \ref{morandilemma}. Assume that $k \ge 1$ and that
  \[
  L_{s^k x}
  =
  \sum_{q=0}^k
  \binom{k}{q}
  L^{k-q}_s
  [ \underbrace{R_s,\dots,R_s}_q, L_x ].
  \]
Using Lemma \ref{morandilemma}, the fact that $R_s$ and $L_s$ commute (see the
proof of Lemma \ref{powerofL}) and Pascal's identity for binomial coefficients,
we obtain
  \allowdisplaybreaks
  \begin{align*}
  L_{s^{k+1}x}
  &=
  L_{s(s^k x)}
  =
  L_s L_{s^kx} + [R_s, L_{s^k x}]
  \\
  &=
  \sum_{q=0}^k
  \binom{k}{q}
  L^{k+1-q}_s
  [ \underbrace{R_s,\dots,R_s}_q, L_x ]
  +
  \sum_{q=0}^k \binom{k}{q}
  L^{k-q}_s
  [ \underbrace{R_s,\dots,R_s}_{q+1}, L_x ]
  \\
  &=
  L^{k+1}_s L_x
  +
  \sum_{q=1}^k
  \binom{k}{q}
  L^{k+1-q}_s
  [ \underbrace{R_s,\dots,R_s}_q, L_x ]
  \\
  &\qquad
  +
  \sum_{q=1}^k
  \binom{k}{q{-}1}
  L^{k+1-q}_s
  [ \underbrace{R_s,\dots, R_s}_q, L_x ]
  +
  [ \underbrace{R_s,\dots, R_s}_{k+1}, L_x ]
  \\
  &=
  L^{k+1}_s L_x
  +
  \sum_{q=1}^k
  \binom{k{+}1}{q}
  L^{k+1-q}_s
  [ \underbrace{R_s,\dots, R_s}_q, L_x ]
  +
  [ \underbrace{R_s,\dots,R_s}_{k+1}, L_x ]
  \\
  &=
  \sum_{q=0}^{k+1}
  \binom{k{+}1}{q}
  L^{k+1-q}_s
  [ \underbrace{R_s,\dots,R_s}_q, L_x ].
  \end{align*}
The proof is complete.
\end{proof}

\begin{proposition}\label{bcde-operator}
In $U(M)$ we have
  \[
  L_{ b^j c^k d^l e^m }
  =
  \sum_{\alpha=0}^{\min(j,k)}
  \alpha!
  \binom{j}{\alpha}
  \binom{k}{\alpha}
  L^{j-\alpha}_b
  L^{k-\alpha}_c
  L^l_d
  L^m_e
  D^{\alpha},
  \qquad
  D = -\tfrac13 ( 2 \, \mathrm{ad}_d + D_{b,c} ).
  \]
\end{proposition}

\begin{proof}
Lemmas \ref{cde-operator} and \ref{bigmorandilemma} give
  \[
  L_{ b^j c^k d^l e^m }
  =
  \sum^j_{\alpha=0}
  \binom{j}{\alpha}
  L^{j-\alpha}_b
  [
  \underbrace{R_b,\dots,R_b}_\alpha,
  L^k_c L^l_d L^m_e
  ].
  \]
By Lemmas \ref{morandilemma} and \ref{zeroderivations} we have
  \[
  [L_d, R_b] = -\tfrac13(\mathrm{ad}_{[d,b]} + D_{d,b}) = 0,
  \qquad
  [L_e, R_b] = -\tfrac13(\mathrm{ad}_{[e,b]} + D_{e,b}) = 0.
  \]
Thus $R_b$ commutes with $L_d$ and $L_e$, and so
  \[
  [ \underbrace{R_b,\dots,R_b}_\alpha, L^k_c L^l_d L^m_e ]
  =
  [ \underbrace{R_b,\dots,R_b}_\alpha, L^k_c] L^l_d L^m_e.
  \]
We have
  \[
  D_{b,c}
  =
  -\mathrm{ad}_{[b,c]} - 3[L_b, R_c]
  =
  -\mathrm{ad}_{[b,c]} - 3[R_b, L_c]
  =
  -2 \, \mathrm{ad}_d - 3 [R_b, L_c].
  \]
Hence $D = [ R_b, L_c ] = R_b L_c - L_c R_b$. We show that $D$ commutes with
$R_b$. For this we compute $[\mathrm{ad}_d, R_b]$ and $[D_{b,c}, R_b]$ using
Lemma \ref{morandilemma}:
  \allowdisplaybreaks
  \begin{align*}
  [ \mathrm{ad}_d, R_b ]
  &=
  [ \mathrm{ad}_d, \mathrm{ad}_b {+} L_b ]
  =
  [ \mathrm{ad}_d, \mathrm{ad}_b ] + [ \mathrm{ad}_d, L_b ]
  \\
  &=
  2 \, D_{d,b} + \mathrm{ad}_{[d,b]} + [R_d {-} L_d, L_b]
  =
  [R_d, L_b] - [L_d, L_b]
  \\
  &=
  [R_d, L_b] - L_{[d,b]} + 2 [R_d, L_b]
  =
  3 [R_d, L_b]= - D_{d,b} = 0,
  \\
  [ D_{b,c}, R_b ] (x)
  &=
  D_{b,c} R_b (x) - R_b D_{b,c} (x)
  =
  D_{b,c}(xb) - D_{b,c}(x)b
  \\
  &=
  D_{b,c}(x) b + x D_{b,c}(b) - D_{b,c}(x) b
  =
  x D_{b,c}(b) = 0.
  \end{align*}
Thus $[D, R_b]=0$. We show that $[ R_b, L^k_c ] = k D L^{k-1}_c$ by induction
on $k$; the case $k = 1$ is clear. By the inductive hypothesis we have
  \allowdisplaybreaks
  \begin{align*}
  [R_b, L^{k+1}_c]
  &=
  R_b L^{k+1}_c - L^{k+1}_c R_b
  =
  ( R_b L^k_c ) L_c - L^{k+1}_c R_b
  \\
  &=
  (L^k_c R_b + k D L^{k-1}_c) L_c - L^{k+1}_c R_b
  =
  L^k_c R_b L_c + k D L^k_c - L^{k+1}_c R_b
  \\
  &=
  L^k_c (L_c R_b + D) + k D L^k_c - L^{k+1}_c R_b
  =
  (k{+}1) D L^k_c,
  \end{align*}
since $D$ and $L_c$ commute:
  \allowdisplaybreaks
  \begin{align*}
  [ D, L_c ] (x)
  &=
  - \tfrac13 [ D_{b,c} + 2 \, \mathrm{ad}_d, L_c ] (x)
  =
  - \tfrac13 [ D_{b,c}, L_c ] (x) - \tfrac23 [ \mathrm{ad}_d, L_c ] (x)
  \\
  &=
  - \tfrac13 \big( D_{b,c} (cx) - c D_{b,c} (x) \big)
  - \tfrac23 \big( [ R_d, L_c ] - [ L_d, L_c] \big) (x)
  \\
  &=
  - \tfrac13 D_{b,c}(c) x = 0.
  \end{align*}
(Recall that $[ R_d, L_c ] = 0$ and $[ L_d, L_c] = 0$ by the proof of Lemma
\ref{cde-operator}.) Finally, we show by induction on $\alpha$ that
  \[
  [ \underbrace{R_b,\dots,R_b}_\alpha, L^k_c]
  =
  \alpha! \binom{k}{\alpha}
  D^{\alpha} L^{k-\alpha}_c.
  \]
This is clear for $\alpha = 0$ and we just proved it for $\alpha = 1$. Since $[
R_b, D ] = 0$ we have
  \allowdisplaybreaks
  \begin{align*}
  [ \underbrace{R_b,\dots,R_b}_{\alpha+1}, L^k_c ]
  &=
  [ R_b, [ \underbrace{R_b\dots,R_b}_\alpha, L^k_c ] ]
  =
  \alpha! \binom{k}{\alpha}
  [ R_b, D^{\alpha} L^{k-\alpha}_c ]
  \\
  &=
  \alpha! \binom{k}{\alpha}
  D^{\alpha} [ R_b, L^{k-\alpha}_c ]
  =
  \alpha! \binom{k}{\alpha} (k{-}\alpha)
  D^{\alpha+1} L^{k-(\alpha+1)}_c
  \\
  &=
  (\alpha{+}1)! \binom{k}{\alpha{+}1}
  D^{\alpha+1} L^{k-(\alpha+1)}_c.
  \end{align*}
The proof is complete.
\end{proof}


\section{Representation of $\mathbb{M}$ by differential operators}
\label{operatorsection}

We now return to the one-parameter family $\mathbb{M} = \mathbb{M}_\gamma$ of
5-dimensional solvable Malcev algebras with structure constants in equation
\eqref{oneparameterfamily}.

\begin{notation} \label{operatorsnotation}
We have these linear operators on the polynomial algebra $P({\mathbb M})$:
  \begin{itemize}
  \item $I$ is the identity;
  \item $M_x$ is multiplication by $x \in \{ a,b,c,d,e\}$;
  \item $D_x$ is differentiation with respect to $x \in \{a,b,c,d,e\}$ (it
      is important to distinguish between this $D_x$ and the previous
      $D_{s,t}$);
  \item $S$ is the shift operator on the generator $a$: $S(a^i b^j c^k
      d^{l} e^m) = ( a + 1 )^i b^j c^k d^{l} e^m$; more generally, we use
      exponential notation and write $S^\alpha$ ($\alpha \in F$) for the
      shift-by-$\alpha$ operator $S_{\alpha} (a^i b^j c^k d^{l} e^m) = ( a
      + \alpha )^i b^j c^k d^{l} e^m$.
  \end{itemize}
\end{notation}

\begin{notation}
We use the linear isomorphism $\phi\colon U(\mathbb{M}) \to P(\mathbb{M})$ to
define operators on $P(\mathbb{M})$ expressing products and commutators in
$U(\mathbb{M})$:
  \begin{itemize}
  \item $L_x$ is left multiplication: $L_x(f) = \phi( x \phi^{-1}(f) )$ for
      $x \in \mathbb{M}$, $f \in P(\mathbb{M})$;
  \item $R_x$ is right multiplication: $R_x(f) = \phi( \phi^{-1}(f) x )$
      for $x \in \mathbb{M}$, $f \in P(\mathbb{M})$;
  \item $\rho_x = R_x - L_x$ is the adjoint: $\rho_x(f) = \phi( [
      \phi^{-1}(f), x ] )$ for $x \in \mathbb{M}$, $f \in P(\mathbb{M})$.
  \end{itemize}
\end{notation}

\begin{lemma} \label{a-formulas}
As operators on $P(\mathbb{M})$ we have
  \allowdisplaybreaks
  \begin{align*}
  L_a
  &=
  M_a,
  \qquad
  \rho_a
  =
  M_b D_b + M_c D_c - M_d D_d + \gamma M_e D_e - 3 D_b D_c M_d,
  \\
  R_a
  &=
  M_a + M_b D_b + M_c D_c - M_d D_d + \gamma M_e D_e -  3D_b D_c M_d.
  \end{align*}
\end{lemma}

\begin{proof}
The claim for $L_a$ is trivial by our convention on basis monomials:
  \[
  L_a ( a^i b^j c^k d^l e^m )
  =
  a ( a^i b^j c^k d^l e^m )
  =
  a^{i+1} b^j c^k d^l e^m.
  \]
Since by definition $R_a = L_a + \rho_a$, it remains only to prove the claim
for $\rho_a$. We first show by induction on $i$ that
  \[
  [ a^i b^j c^k d^l e^m, a ]
  =
  a^i [ b^j c^k d^l e^m, a ].
  \]
For $i = 0$ the claim is trivial.  For the inductive step, equation
\eqref{rightbracket} gives
  \allowdisplaybreaks
  \begin{align*}
  &
  [ a^{i+1} b^j c^k d^l e^m, a ]
  =
  [ a a^i b^j c^k d^l e^m, a ]
  \\
  &=
  [a,a] a^i b^j c^k d^l e^m
  +
  a [ a^i b^j c^k d^l e^m, a ]
  -
  D_{a,a} ( a^i b^j c^k d^l e^m )
  -
  [ a^i b^j c^k d^l e^m, [a,a] ]
  \\
  &=
  a [ a^i b^j c^k d^l e^m, a ].
  \end{align*}
By definition of $M_x$ and $D_x$ we have
  \allowdisplaybreaks
  \begin{align*}
  &
  ( M_b D_b + M_c D_c - M_d D_d + \gamma M_e D_e - 3 D_b D_c M_d )
  ( b^j c^k d^l e^m )
  \\
  &=
  ( j+k-l+m \gamma ) b^j c^k d^l e^m
  -
  3 j k b^{j-1} c^{k-1} d^{l+1} e^m,
  \end{align*}
and so it now remains only to show by induction on $j$ that
  \[
  [ b^j c^k d^l e^m, a ]
  =
  ( j+k-l+m \gamma ) b^j c^k d^l e^m
  -
  3 j k b^{j-1} c^{k-1} d^{l+1} e^m.
  \]
For $j = 0$ we use the fact that $a$, $c$, $d$, $e$ span a Lie subalgebra of
$\mathbb{M}$ (a nilpotent Lie algebra, the split extension of the 1-dimensional
Lie algebra with basis $a$ by the module with basis $c$, $d$, $e$). Clearly
$\rho_a$ is a derivation of this Lie algebra, and we have $[x,[x,a]] = 0$ for
$x  = c, d, e$. Therefore
  \[
  [ c^k d^l e^m, a ]
  =
  [ c^k, a ] d^l e^m + c^k [ d^l, a ] e^m + c^k d^l [ e^m, a ]
  =
  ( k - l + m \gamma ) c^k d^l e^m.
  \]
For the inductive step, equation \eqref{rightbracket} gives
  \allowdisplaybreaks
  \begin{align*}\
  &
  [ b^{j+1} c^k d^l e^m, a ]
  =
  [ b b^j c^k d^l e^m, a ]
  \\
  &=
  [b,a] b^j c^k d^l e^m
  +
  b [ b^j c^k d^l e^m, a ]
  -
  D_{a,b}( b^j  c^k d^l e^m )
  -
  [ b^j c^k d^l e^m, [a,b] ]
  \\
  &=
  b^{j+1} c^k d^l e^m
  +
  b [ b^j c^k d^l e^m, a ]
  -
  D_{a,b}( b^j  c^k d^l e^m )
  +
  [ b^j c^k d^l e^m, b ].
  \end{align*}
Since $b$, $c$, $d$, $e$ span a Lie subalgebra of $\mathbb{M}$, the structure
constants give
  \[
  [ b^j c^k d^l e^m, b ]
  =
  - 2k b^j c^{k-1} d^{l+1} e^m.
  \]
Furthermore, since $D_{a,b}(b) = 0$, $D_{a,b}(c) = d$, $D_{a,b}(d) = 0$,
$D_{a,b}(e) = 0$, we obtain
  \allowdisplaybreaks
  \begin{align*}
  &
  D_{a,b}( b^j  c^k d^l e^m )
  \\
  &=
  D_{a,b}( b^j ) c^k d^l e^m
  +
  b^j D_{a,b}( c^k ) d^l e^m
  +
  b^j  c^k D_{a,b}( d^l ) e^m
  +
  b^j  c^k d^l D_{a,b}( e^m )
  \\
  &=
  b^j D_{a,b}( c^k ) d^l e^m
  =
  k b^j c^{k-1} d^{l+1} e^m.
  \end{align*}
Combining these results gives
  \[
  [ b^{j+1} c^k d^l e^m, a ]
  =
  b^{j+1} c^k d^l e^m
  +
  b [ b^j c^k d^l e^m, a ]
  -
  3 k
  b^j c^{k-1} d^{l+1} e^m.
  \]
Now the inductive hypothesis gives
  \allowdisplaybreaks
  \begin{align*}
  &
  [ b^{j+1} c^k d^l e^m, a ]
  =
  b^{j+1} c^k d^l e^m
  +
  b [ b^j c^k d^l e^m, a ]
  -
  3 k
  b^j c^{k-1} d^{l+1} e^m
  \\
  &=
  b^{j+1} c^k d^l e^m
  +
  b
  \big(
  ( j{+}k{-}l{+}m \gamma )
  b^j c^k d^l e^m - 3 j k b^{j-1} c^{k-1} d^{l+1} e^m
  \big)
  \\
  &\quad
  -
  3 k
  b^j c^{k-1} d^ {l+1} e^m
  \\
  &=
  ( j+1+k-l+m \gamma )  b^{j+1} c^k d^l e^m
  -
  3 (j+1) k b^{j} c^{k-1} d^{l+1} e^m.
  \end{align*}
The proof is complete.
\end{proof}

\begin{lemma} \label{b-formulas}
As operators on $P(\mathbb{M})$ we have
  \allowdisplaybreaks
  \begin{alignat*}{2}
  L_b
  &=
  S M_b +  (  S^{-1} - S ) D_c M_d,
  &\qquad
  \rho_b
  &=
  ( I - S ) M_b + ( S - 2 S^{-1} -  I ) D_c M_d,
  \\
  R_b
  &=
  M_b - ( S^{-1} + I ) D_c M_d,
  &\qquad
  D_{a,b}
  &=
  ( I - S ) M_b + ( S + S^{-1} - I ) D_c M_d.
  \end{alignat*}
\end{lemma}

\begin{proof}
Since $R_b = L_b + \rho_b$, it suffices to prove the formulas for $L_b$,
$\rho_b$ and $D_{a,b}$. For this we set $y = a^i b^j c^k d^l e^m$ and do
simultaneous induction on $i$ using the equations
  \allowdisplaybreaks
  \begin{alignat*}{2}
  b(ay)
  &=
  a(by) + by - \tfrac23 D_{a,b}(y) + \tfrac23 [y,b],
  &\qquad
  D_{a,b}(ay)
  &=
  - by + a D_{a,b}(y),
  \\
  [ ay, b ]
  &=
  -by + a[y,b] + D_{a,b}(y) - [y,b],
  \end{alignat*}
which follow from \eqref{rightbracket}, \eqref{leftproduct} and Table
\ref{derivationtable}. The basis of the induction consists of these equations
from the proof of Lemma \ref{a-formulas}:
  \allowdisplaybreaks
  \begin{alignat*}{2}
  L_b ( b^j c^k d^l e^m )
  &=
  b^{j+1} c^k d^l e^m,
  &\qquad
  D_{a,b}( b^j  c^k d^l e^m )
  &=
  k b^j c^{k-1} d^{l+1} e^m,
  \\
  \rho_b( b^j c^k d^l e^m)
  &=
  - 2k b^j c^{k-1} d^{l+1} e^m.
  \end{alignat*}
We now prove case $i+1$ of each equation separately, but in each case the
inductive hypothesis is case $i$ of all three equations.  First, the formula
for $L_b$:
  \allowdisplaybreaks
  \begin{align*}
  &
  b ( a^{i+1} b^j c^k d^l e^m )
  \\
  &=
  a(b ( a^i b^j c^k d^l e^m ) )
  +
  b( a^i b^j c^k d^l e^m )
  -
  \tfrac23
  D_{a,b}( a^i b^j c^k d^l e^m )
  +
  \tfrac23
  [ a^i b^j c^k d^l e^m, b ]
  \\
  &=
  ( M_a + I )
  ( S M_b + ( S^{-1} - S ) D_c M_d )
  ( a^i b^j c^k d^l e^m )
  \\
  &\quad
  -
  \tfrac23
  (
  ( I - S ) M_b + ( S + S^{-1} - I ) D_c M_d
  )
  ( a^i b^j c^k d^l e^m )
  \\
  &\quad
  +
  \tfrac23
  (
  ( I - S ) M_b + ( S - 2 S^{-1} -  I ) D_c M_d
  )
  ( a^i b^j c^k d^l e^m )
  \\
  &=
  ( M_a + I ) ( S M_b )
  ( a^i b^j c^k d^l e^m )
  \\
  &\quad
  +
  (
  ( M_a {+} I ) ( S^{-1} {-} S )
  -
  \tfrac23
  ( S {+} S^{-1} {-} I )
  +
  \tfrac23
  ( S {-} 2S^{-1} {-} I )
  )
  ( D_c M_d )
  ( a^i b^j c^k d^l e^m )
  \\
  &=
  ( S M_a M_b )
  ( a^i b^j c^k d^l e^m )
  +
  ( (M_a - I) S^{-1} - (M_a + I) S )
  ( D_c M_d )
  ( a^i b^j c^k d^l e^m )
  \\
  &=
  ( S M_a M_b )
  ( a^i b^j c^k d^l e^m )
  +
  ( S^{-1} M_a - S M_a )
  ( D_c M_d )
  ( a^i b^j c^k d^l e^m )
  \\
  &=
  ( S M_b M_a )
  ( a^i b^j c^k d^l e^m )
  +
  ( ( S^{-1} - S ) M_a )
  ( D_c M_d )
  ( a^i b^j c^k d^l e^m )
  \\
  &=
  (
  S M_b +  (  S^{-1} - S ) D_c M_d
  )
  ( a^{i+1} b^j c^k d^l e^m ).
  \end{align*}
Next, the formula for $\rho_b$:
  \allowdisplaybreaks
  \begin{align*}
  &
  [ a^{i+1} b^j c^k d^l e^m , b ]
  \\
  &=
  -
  b ( a^i b^j c^k d^l e^m )
  +
  a [ a^i b^j c^k d^l e^m, b ]
  +
  D_{a,b} ( a^i b^j c^k d^l e^m )
  -
  [ a^i b^j c^k d^l e^m, b ]
  \\
  &=
  -
  ( S M_b + ( S^{-1} - S ) D_c M_d )
  ( a^i b^j c^k d^l e^m )
  \\
  &\quad
  +
  ( M_a - I )
  (
  ( I - S ) M_b + ( S - 2 S^{-1} - I ) D_c M_d
  )
  ( a^i b^j c^k d^l e^m )
  \\
  &\quad
  +
  (
  ( I - S ) M_b + ( S + S^{-1} - I ) D_c M_d
  )
  ( a^i b^j c^k d^l e^m )
  \\
  &=
  ( - S + M_a (I - S) ) M_b
  ( a^i b^j c^k d^l e^m )
  \\
  &\quad
  +
  ( ( M_a + I ) S - 2 ( M_a - I ) S^{-1} - M_a )
  ( D_c M_d )
  (a^i b^j c^k d^l e^m)
  \\
  &=
  ( M_a - ( M_a + I ) S ) M_b
  ( a^i b^j c^k d^l e^m )
  \\
  &\quad
  +
  ( ( M_a + I ) S - 2 ( M_a - I ) S^{-1} - M_a )
  ( D_c M_d )
  ( a^i b^j c^k d^l e^m )
  \\
  &=
  ( M_a - S M_a ) M_b
  ( a^i b^j c^k d^l e^m )
  +
  ( S M_a - 2 S^{-1} M_a - M_a ) ( D_c M_d )
  ( a^i b^j c^k d^l e^m )
  \\
  &=
  ( I - S ) ( M_a M_b )
  ( a^i b^j c^k d^l e^m )
  +
  ( S M_a - 2 S^{-1} M_a - M_a )
  ( D_c M_d )
  ( a^i b^j c^k d^l e^m )
  \\
  &=
  ( I - S ) ( M_b M_a )
  ( a^i b^j c^k d^l e^m )
  +
  ( S - 2 S^{-1} - I ) ( M_a D_c M_d )
  ( a^i b^j c^k d^l e^m )
  \\
  &=
  (
  ( I - S ) M_b + ( S - 2 S^{-1} - I ) D_c M_d
  )
  ( a^{i+1} b^j c^k d^l e^m ).
  \end{align*}
Finally, the formula for $D_{a,b}$:
  \allowdisplaybreaks
  \begin{align*}
  &
  D_{a,b} ( a^{i+1} b^j c^k d^l e^m )
  =
  - b ( a^i b^j c^k d^l e^m ) + a D_{a,b} ( a^i b^j c^k d^l e^m )
  \\
  &=
  -
  ( S M_b + ( S^{-1} - S ) D_c M_d )
  ( a^i b^j c^k d^l e^m )
  \\
  &\quad
  +
  M_a ( ( I - S ) M_b + ( S + S^{-1} - I ) D_c M_d )
  ( a^i b^j c^k d^l e^m )
  \\
  &=
  ( - S + M_a (I - S ) ) M_b
  ( a^i b^jc^k d^l e^m )
  \\
  &\quad
  +
  ( ( I + M_a ) S + ( M_a - I ) S^{-1} - M_a ) ( D_c M_d )
  ( a^i b^j c^k d^l e^m )
  \\
  &=
  ( M_a - (I + M_a) S ) M_b
  ( a^i b^jc^k d^l e^m )
  \\
  &\quad
  +
  ( ( I + M_a ) S + ( M_a - I ) S^{-1} - M_a ) ( D_c M_d )
  ( a^i b^j c^k d^l e^m )
  \\
  &=
  ( M_a - S M_a ) M_b
  ( a^i b^jc^k d^l e^m )
  +
  ( S M_a + S^{-1} M_a - M_a ) ( D_c M_d )
  ( a^i b^j c^k d^l e^m )
  \\
  &=
  ( I - S ) ( M_a M_b )
  ( a^i b^j c^k d^l e^m )
  +
  ( S + S^{-1} - I ) ( M_a D_c M_d )
  ( a^i b^j c^k d^l e^m )
  \\
  &=
  ( ( I - S ) M_b + ( S + S^{-1} - I ) D_c M_d )
  ( a^{i+1} b^j c^k d^l e^m ).
  \end{align*}
The proof is complete.
\end{proof}

\begin{lemma} \label{c-formulas}
As operators on $P(\mathbb{M})$ we have
  \allowdisplaybreaks
  \begin{alignat*}{2}
  L_c
  &=
  S M_c - ( S + S^{-1} ) D_b M_d,
  &\qquad
  \rho_c
  &=
  ( I - S ) M_c + ( S + 2 S^{-1} - I ) D_b M_d,
  \\
  R_c
  &=
  M_c + ( S^{-1} - I ) D_b M_d,
  &\qquad
  D_{a,c}
  &=
  ( I - S ) M_c + (  S - S^{-1} - I  ) D_b M_d.
  \end{alignat*}
\end{lemma}

\begin{proof}
As before, it suffices to prove the formulas for $L_c$, $\rho_c$ and $D_{a,c}$.
The basis of the induction consists of the following equations which follow
easily from the fact that $b$, $c$, $d$, $e$ span a Lie subalgebra of
$\mathbb{M}$:
  \allowdisplaybreaks
  \begin{alignat*}{2}
  &
  c ( b^j c^k d^l e^m )
  =
  b^j c^{k+1} d^l e^m - j b^{j-1} c^k d^{l+1} e^m,
  &\qquad
  &
  [ b^j c^k d^l e^m, c ]
  =
  2 j b^{j-1} c^k d^{l+1} e^m,
  \\
  &
  D_{a,c}( b^j c^k d^l e^m )
  =
  - j b^{j-1} c^k d^{l+1} e^m.
  \end{alignat*}
The strategy of the proof is the same as for Lemma \ref{b-formulas}. First, the
formula for $L_c$:
  \allowdisplaybreaks
  \begin{align*}
  &
  c ( a^{i+1} b^j c^k d^l e^m )
  \\
  &=
  a( c ( a^i b^j c^k d^l e^m ) )
  +
  c ( a^i b^j c^k d^l e^m )
  -
  \tfrac23 D_{a,c}( a^i b^j c^k d^l e^m )
  +
  \tfrac23 [ a^i b^j c^k d^l e^m, c ]
  \\
  &=
  ( M_a + I )
  (
  S M_c -( S + S^{-1} ) D_b M_d
  )
  ( a^i b^j c^k d^l e^m )
  \\
  &\quad
  -
  \tfrac23
  (
  ( I - S ) M_c + ( S - S^{-1} - I ) D_b M_d
  )
  ( a^i b^j c^k d^l e^m )
  \\
  &\quad
  +
  \tfrac23
  (
  ( I - S ) M_c + ( S + 2 S^{-1} - I ) D_b M_d
  )
  ( a^i b^j c^k d^l e^m )
  \\
  &=
  (
  ( M_a + I ) ( S M_c )
  -
  ( ( M_a + I ) ( S + S^{-1} ) - 2 S^{-1} )
  ( D_b M_d )
  )
  ( a^i b^j c^k d^l e^m )
  \\
  &=
  (
  ( M_a + I ) ( S M_c )
  -
  ( ( M_a + I ) S + ( M_a - I ) S^{-1} ) ( D_b M_d )
  )
  ( a^i b^j c^k d^l e^m )
  \\
  &=
  (
  ( S M_a ) M_c - ( S M_a + S^{-1} M_a ) D_b M_d
  )
  ( a^i b^j c^k d^l e^m )
  \\
  &=
  (
  ( S M_c ) M_a - ( S + S^{-1} ) M_a D_b M_d
  )
  ( a^i b^j c^k d^l e^m )
  \\
  &=
  ( S M_c - ( S + S^{-1} ) D_b M_d )
  ( a^{i+1} b^j c^k d^l e^m ).
  \end{align*}
Next, the formula for $\rho_c$:
  \allowdisplaybreaks
  \begin{align*}
  &
  [ a^{i+1} b^j c^k d^l e^m, c ]
  \\
  &=
  -
  c (a^i b^j c^k d^l e^m)
  +
  a [ a^i b^j c^k d^l e^m, c ]
  +
  D_{a,c} ( a^i b^j c^k d^l e^m )
  -
  [ a^i b^j c^k d^l e^m, c ]
  \\
  &=
  ( -S M_c + ( S + S^{-1} ) D_b M_d )
  ( a^i b^j c^k d^l e^m )
  \\
  &\quad
  +
  ( M_a - I )
  ( ( I - S ) M_c + ( S + 2 S^{-1} - I ) D_b M_d )
  ( a^i b^j c^k d^l e^m )
  \\
  &\quad
  +
  ( ( I - S ) M_c + ( S - S^{-1} - I ) D_b M_d )
  ( a^i b^j c^k d^l e^m )
  \\
  &=
  ( - S + M_a ( I - S ) ) M_c
  (a^i b^j c^k d^l e^m)
  \\
  &\quad
  +
  ( S + S^{-1} + ( M_a - I ) ( S {+} 2 S^{-1} {-} I ) + ( S {-} S^{-1} {-} I ) )
  ( D_b M_d )
  (a^i b^j c^k d^l e^m)
  \\
  &=
  ( M_a - ( M_a + I ) S ) M_c
  ( a^i b^j c^k d^l e^m )
  \\
  &\quad
  +
  ( ( M_a + I ) S + 2 ( M_a - I ) S^{-1} - M_a )
  ( D_b M_d )
  ( a^i b^j c^k d^l e^m )
  \\
  &=
  ( M_a - S M_a ) M_c
  ( a^i b^j c^k d^l e^m )
  +
  ( S M_a + 2 S^{-1} M_a - M_a )
  ( D_b M_d )
  ( a^i b^j c^k d^l e^m )
  \\
  &=
  ( I - S ) ( M_a M_c )
  ( a^i b^j c^k d^l e^m )
  +
  ( S + 2 S^{-1} - I )
  ( M_a D_b M_d )
  ( a^i b^j c^k d^l e^m )
  \\
  &=
  (
  ( I - S ) M_c + ( S + 2 S^{-1} - I ) ( D_b M_d )
  )
  (a^{i+1} b^j c^k d^l e^m).
  \end{align*}
Finally, the formula for $D_{a,c}$:
  \allowdisplaybreaks
  \begin{align*}
  &
  D_{a,c} (a^{i+1} b^j c^k d^l e^m)
  =
  - c ( a^i b^j c^k d^l e^m )
  +
  a D_{a,c} ( a^i b^j c^k d^l e^m )
  \\
  &=
  -
  ( S M_c - ( S + S^{-1} ) D_b M_d )
  ( a^i b^j c^k d^l e^m )
  \\
  &\quad
  +
  M_a
  ( ( I - S ) M_c + ( S - S^{-1} - I ) D_b M_d )
  ( a^i b^j c^k d^l e^m )
  \\
  &=
  (
  ( - S {+} M_a ( I {-} S ) ) M_c
  +
  ( S {+} S^{-1} {+} M_a S {-} M_a S^{-1} {-} M_a )
  ( D_b M_d )
  )
  ( a^i b^j c^k d^l e^m )
  \\
  &=
  ( ( ( I - S ) M_a ) M_c + ( S M_a - S^{-1} M_a - M_a ) D_b M_d )
  ( a^i b^j c^k d^l e^m )
  \\
  &=
  \left(
  ( I - S ) M_c
  +
  (  S - S^{-1} - I  ) D_b M_d
  \right)
  ( a^{i+1} b^j c^k d^l e^m ).
  \end{align*}
The proof is complete.
\end{proof}

\begin{lemma} \label{d-formulas}
As operators on $P(\mathbb{M})$ we have
  \[
  L_d
  =
  S^{-1} M_d,
  \qquad
  \rho_d
  =
  ( I - S^{-1} ) M_d,
  \qquad
  R_d
  =
  M_d,
  \qquad
  D_{a,d}
  =
  ( S^{-1} - I ) M_d.
  \]
\end{lemma}

\begin{proof}
The basis of the induction consists of the equations
  \[
  d ( b^j c^k d^l e^m ) = b^j c^k d^{l+1} e^m,
  \qquad
  [ b^j c^k d^l e^m, d ] = 0,
  \qquad
  D_{a,d} ( b^j c^k d^l e^m ) = 0.
  \]
The rest of the proof is similar to that of Lemma \ref{c-formulas}.
\end{proof}

\begin{lemma} \label{e-formulas}
As operators on $P(\mathbb{M})$ we have
  \[
  L_e
  =
  S^{-\gamma} M_e,
  \qquad
  \rho_e
  =
  ( I - S^{-\gamma} ) M_e,
  \qquad
  R_e
  =
  M_e,
  \qquad
  D_{a,e}
  =
  \gamma ( S^{-\gamma} - I ) M_e.
  \]
\end{lemma}

\begin{proof}
The basis of the induction consists of the three equations
  \[
  D_{a,e}( b^j c^k d^l e^m ) = 0,
  \qquad
  e ( b^j c^k d^l e^m ) = b^j c^k d^l e^{m+1},
  \qquad
  [ b^j c^k d^l e^m, e ] = 0.
  \]
Recall that $S^\gamma$ is defined by $S^\gamma( a^i b^j c^k d^l e^m ) = ( a +
\gamma )^i b^j c^k d^l e^m$. We have
  \allowdisplaybreaks
  \begin{align*}
  &
  e (a^{i+1} b^j c^k d^l e^m)
  \\
  &=
  a ( e (a^i b^j c^k d^l e^m) )
  -
  \gamma
  e ( a^i b^j c^k d^l e^m )
  -
  \tfrac23
  D_{a,e} ( a^i b^j c^k d^l e^m )
  -
  \tfrac23 \gamma
  [ a^i b^j c^k d^l e^m, e ]
  \\
  &=
  (
  ( M_a - \gamma I ) S^{-\gamma} M_e
  -
  \tfrac23 \gamma
  ( S^{-\gamma} - I ) M_e
  -
  \tfrac23 \gamma
  ( I - S^{-\gamma} ) M_e
  )
  ( a^i b^j c^k d^l e^m )
  \\
  &=
  S^{-\gamma} M_a M_e (a^i b^j c^k d^l e^m)
  =
  S^{-\gamma} M_e (a^{i+1} b^j c^k d^l e^m),
  \\
  &
  [a^{i+1} b^j c^k d^l e^m, e]
  \\
  &=
  \gamma e(a^i b^j c^k d^l e^m)
  +
  a [ a^i b^j c^k d^l e^m, e ]
  +
  D_{a,e} ( a^i b^j c^k d^l e^m )
  +
  \gamma
  [ a^i b^j c^k d^l e^m, e ]
  \\
  &=
  (
  \gamma
  S^{-\gamma} M_e
  +
  M_a ( I {-} S^{-\gamma} ) M_e
  +
  \gamma
  ( S^{-\gamma} {-} I ) M_e
  +
  \gamma
  ( I {-} S^{-\gamma} ) M_e
  )
  ( a^i b^j c^k d^l e^m )
  \\
  &=
  ( \gamma S^{-\gamma} {+} M_a {-} M_a S^{-\gamma} ) M_e
  (a^i b^j c^k d^l e^m)
  =
  ( - ( M_a {-} \gamma I ) S^{-\gamma} {+} M_a ) M_e
  ( a^i b^j c^k d^l e^m )
  \\
  &=
  ( - S^{-\gamma} + I ) M_a M_e
  ( a^i b^j c^k d^l e^m )
  =
  ( I - S^{-\gamma} ) M_e
  ( a^{i+1} b^j c^k d^l e^m ),
  \\
  &
  D_{a,e} ( a^{i+1} b^j c^k d^l e^m )
  =
  -
  \gamma^2 e ( a^i b^j c^k d^l e^m )
  +
  a D_{a,e} ( a^i b^j c^k d^l e^m )
  \\
  &=
  -\gamma^2 S^{-\gamma}
  M_e ( a^i b^j c^k d^l e^m )
  +
  M_a \gamma ( S^{-\gamma} - I ) M_e
  ( a^i b^j c^k d^l e^m )
  \\
  &=
  \gamma
  ( - \gamma S^{-\gamma} M_e + M_a S^{-\gamma} M_e - M_a M_e )
  ( a^i b^j c^k d^l e^m )
  \\
  &=
  \gamma
  ( ( M_a - \gamma I ) S^{-\gamma} - M_a ) M_e
  ( a^i b^j c^k d^l e^m )
  =
  \gamma
  ( ( S^{-\gamma} M_a - M_a ) M_e )
  ( a^i b^j c^k d^l e^m )
  \\
  &=
  \gamma
  ( S^{-\gamma} - I) M_a M_e
  ( a^i b^j c^k d^l e^m )
  =
  \gamma
  ( S^{-\gamma} - I) M_e
  ( a^{i+1} b^j c^k d^l e^m ).
  \end{align*}
The proof is complete.
\end{proof}


\section{The center of the universal enveloping algebra}

Our next goal is to use the results of Section \ref{operatorsection} to compute
the center of $U({\mathbb M})$.  Let $Z(U)$ and $K(U)$ denote the center and
commutative center of $U(\mathbb{M})$. Shestakov and Zhelyabin
\cite{ShestakovZhelyabin} proved that
  \[
  Z(U) = K(U)
  =
  \{ \,
  n \in U(\mathbb{M})
  \mid
  [n,x] = 0 \, \text{for all} \, x\in \mathbb{M}
  \, \}.
  \]
$Z(U)$ is a characteristic subalgebra: it is stable under automorphisms of
$U(\mathbb{M})$.

\begin{theorem}
Let $\mathbb{M} = \mathbb{M}_\gamma$ belong to the one-parameter family
\eqref{oneparameterfamily} of solvable 5-dimensional Malcev algebras over a
field $F$. Then
  \[
  Z(U)
  =
  \begin{cases}
  F [ \, d^{\gamma m} e^m \, ]
  &\text{if $\gamma = l/m$ with $l, m \in \mathbb{Z}$, $(l,m) = 1$, $m > 0$};
  \\
  F
  &\text{if $\gamma \notin \mathbb{Q}$}.
  \end{cases}
  \]
\end{theorem}

\begin{proof}
We first show that $Z(U) \subseteq U(L)$ where $L = \text{span}\{ b, c, d, e
\}$. Choose
  \[
  n = \sum_{i=0}^m a^i s_i \in Z(U),
  \quad
  s_i \in U(L),
  \]
where $m$ is minimal satisfying $n \notin U(L)$ and $s_m \ne 0$. If $m = 0$
then we are done. If $m \ge 1$ then we show by contradiction that $m = 1$.
Assume $m > 1$ and consider the automorphism $\varphi$ of $U(\mathbb{M})$
defined by
  \[
  \varphi(a) = a+1, \quad
  \varphi(b) = b, \quad
  \varphi(c) = c, \quad
  \varphi(d) = d, \quad
  \varphi(e) = e.
  \]
Since $Z(U)$ is a characteristic subalgebra, $\varphi(n) \in Z(U)$ and so
$\varphi(n) - n \in Z(U)$. However,
  \[
  \varphi(n)-n
  =
  \sum_{i=0}^m (a+1)^i s_i - \sum_{i=0}^m a^i s_i
  =
  \sum_{i=0}^m ( (a+1)^i - a^i ) s_i.
  \]
Hence $\varphi(n) - n$ is a nonzero element of $Z(U) \setminus U(L)$, but its
degree in $a$ is strictly less than $m$, contradicting the choice of $n$. Hence
$m = 1$ and $n = s_0 + a s_1$. Since $n \in Z(U)$ we have $\mathrm{ad}_d(n) =
[n,d] = 0$. Lemma \ref{d-formulas} shows that $\mathrm{ad}_d = - D_{a,d}$ and
so $\mathrm{ad}_d$ is a derivation of $U(\mathbb{M})$; also $\mathrm{ad}_d = (
I - S^{-1} ) M_d$ which is zero on $U(L)$. Hence
  \[
  0
  =
  \mathrm{ad}_d (n)
  =
  \mathrm{ad}_d (s_0 + a s_1)
  =
  \mathrm{ad}_d(s_0) + \mathrm{ad}_d(a) s_1 + a \mathrm{ad}_d(s_1)
  =
  d s_1.
  \]
Hence $s_1 = 0$ since $U(\mathbb{M})$ has no zero divisors by
\cite{PerezIzquierdoShestakov}. Therefore $Z(U) \subseteq U(L)$.

Clearly $n \in Z(U)$ if and only if $\mathrm{ad}_s(n) = 0$ for $s \in \{ a, b,
c, d, e \}$. Consider
  \[
  n
  =
  \sum_i \alpha_i b^{j_i} c^{k_i} d^{l_i} e^{m_i}
  \in
  Z(U),
  \quad
  \alpha_i \ne 0.
  \]
By the formula for $\rho_b$ from Lemma \ref{b-formulas} we have
  \allowdisplaybreaks
  \begin{align*}
  \mathrm{ad}_b (n)
  &=
  (I - S) M_b (n)
  +
  (S - 2 S^{-1} - I) D_c M_d (n)
  =
  0 - 2 D_c M_d (n)
  \\
  &=
  \sum_i ( -2 \alpha_i k_i ) b^{j_i} c^{k_i-1} d^{l_i+1} e^{m_i};
  \end{align*}
hence $k_i = 0$ for all $i$. By the formula for $\rho_c$ from Lemma
\ref{c-formulas} we have
  \allowdisplaybreaks
  \begin{align*}
  \mathrm{ad}_b (n)
  &=
  (I - S) M_c (n)
  +
  (S + 2 S^{-1} - I) D_b M_d (n)
  =
  0 + 2 D_b M_d (n)
  \\
  &=
  \sum_i ( 2 \alpha_i j_i ) b^{j_i-1} d^{l_i+1} e^{m_i};
  \end{align*}
hence $j_i = 0$ for all $i$. By the formula for $\rho_a$ from Lemma
\ref{a-formulas} we have
  \allowdisplaybreaks
  \begin{align*}
  \mathrm{ad}_a(n)
  =
  \sum_i \alpha_i ( - l_i {+} \gamma m_i ) d^{l_i} e^{m_i};
  \end{align*}
hence $l_i = \gamma m_i $ for all $i$. It is clear that $\mathrm{ad}_d (
d^{l_i} e^{m_i} ) = \mathrm{ad}_e (d^{l_i} e^{m_i}) = 0$.

Hence if $\gamma \in \mathbb{Q}$ then $Z(U)$ is generated by $d^{\gamma m} e^m$
where $m$ is the smallest positive integer for which $\gamma m \in \mathbb{Z}$,
and if $\gamma \notin \mathbb{Q}$ then $Z(U) = F$.
\end{proof}


\section{The universal nonassociative enveloping algebra}

In this section we compute structure constants for $U(\mathbb{M})$ where
$\mathbb{M} = \mathbb{M}_\gamma$ belongs to the one-parameter family
\eqref{oneparameterfamily} of solvable 5-dimensional Malcev algebras.

\begin{lemma} \label{bcde-operatorforM}
In $U(\mathbb{M})$ we have
  \[
  L_{b^j c^k d^l e^m} = \sum_{\alpha=0}^{\min(j,k)}
  \sum^{\alpha}_{\beta=0} (-1)^{\alpha-\beta} \alpha!
  \binom{\alpha}{\beta} \binom{j} {\alpha} \binom{k}{\alpha}
  S^{-\beta} L^{j-\alpha}_b L^{k-\alpha}_c M^{\alpha}_d L^l_d L^m_e.
  \]
\end{lemma}

\begin{proof}
By Lemma \ref{d-formulas} we have $\mathrm{ad}_d = -D_{a,d} = (I-S^{-1})M_d$,
so $\mathrm{ad}_d$ is a derivation of $U(\mathbb{M})$. Since $\mathrm{ad}_d(a)
= d$ and $D_{b,c}(a) = d$, we have $(D_{b,c} - \mathrm{ad}_d)(a) = 0$. Since
$D_{b,c} - \mathrm{ad}_d$ is a derivation of $U(\mathbb{M})$, we have $(D_{b,c}
- \mathrm{ad}_d)(a^k) = 0$. It follows from Table \ref{derivationtable} and
Lemma \ref{d-formulas} that $\mathrm{ad}_d(x)= 0$ and $D_{b,c}(x) = 0$ for any
$x \in U(L)$ where $L = \mathrm{span}\{b, c, d, e\}$. Hence
  \[
  (D_{b,c} - \mathrm{ad}_d) (a^k x)
  =
  (D_{b,c} - \mathrm{ad}_d)(a^k) x
  +
  a^k (D_{b,c} - \mathrm{ad}_d)(x)
  =
  0.
  \]
Therefore $\mathrm{ad}_d = D_{b,c}$ on $U(\mathbb{M})$. This implies that the
operator $D$ from the proof of Proposition \ref{bcde-operator} satisfies $D =
-\mathrm{ad}_d = (S^{-1}-I) M_d$.  Therefore
  \[
  D^{\alpha}
  =
  (S^{-1} - I)^{\alpha} M^{\alpha}_d
  =
  \sum^{\alpha}_{\beta=0}
  (-1)^{\alpha-\beta}
  \binom{\alpha}{\beta}
  S^{-\beta}
  M^{\alpha}_d.
  \]
Using this in Proposition \ref{bcde-operator} gives the stated formula for
$L_{b^j c^k d^l e^m}$.
\end{proof}

\begin{remark}
If we set $m = 0$ in Lemma \ref{bcde-operatorforM} then we obtain the formula
for $L_{b^j c^k d^l}$ in Lemma 4.2 of \cite{BHPU}. The following Lemma
\ref{Ra-bracket} generalizes Lemma 4.3 of \cite{BHPU}.
\end{remark}

\begin{lemma} \label{Ra-bracket}
In $U(\mathbb{M})$ we have
  \allowdisplaybreaks
  \begin{align*}
  &
  [R_a, L_a^s S^t L_b^u D_b^v D_c^w L_c^x M_d^y L_d^z L_e^m]
  =
  -(t+v+w+y) L_a^s S^t L_b^u D_b^v D_c^w L_c^x M_d^y L_d^z L_e^m
  \\
  &
  - u L_a^s S^{t-1} L_b^{u-1} D_b^v D_c^{w+1} L_c^x M_d^{y+1}
  L_d^z L_e^m
  + x L_a^s S^{t-1} L_b^u D_b^{v+1} D_c^w L_c^{x-1}
  M_d^{y+1} L_d^z L_e^m.
  \end{align*}
\end{lemma}

\begin{proof}
For any $x,y\in M$, by Lemma \ref{morandilemma} we have
  \[
  D_{x,y}
  =
  - \mathrm{ad}_{[x,y]} - 3 [ R_x, L_y],
  \quad \text{hence} \quad
  -3 [R_x, L_y]
  =
  \mathrm{ad}_{[x,y]} + D_{x,y}.
  \]
Using this and Lemma \ref{b-formulas} we show that $[ R_a, L_b] = - S^{-1} D_c
M_d$:
  \allowdisplaybreaks
  \begin{align*}
  &
  -3 [R_a, L_b]
  =
  \mathrm{ad}_{[a,b]} + D_{a,b}
  =
  - \mathrm{ad}_b + D_{a,b}
  \\
  &=
  - (I-S) M_b - (S - 2S^{-1} - I) D_c M_d
  + (I-S) M_b + (S+ S^{-1} - I) D_c M_d
  \\
  &=
  3 S^{-1} D_c M_d.
  \end{align*}
Similarly, using Lemma \ref{c-formulas} we show that $[ R_a, L_c] = S^{-1} D_b
M_d$:
  \allowdisplaybreaks
  \begin{align*}
  &
  -3 [ R_a, L_c]
  =
  \mathrm{ad}_{[a,c]} + D_{a,c}
  =
  -\mathrm{ad}_c + D_{a,c}
  \\
  &=
  - (I-S)M_c - (S + 2 S^{-1} - I) D_b M_d
  + (I-S) M_c + (S- S^{-1} - I) D_b M_d
  \\
  &=
  -3 S^{-1} D_b M_d.
  \end{align*}
Using Lemmas \ref{d-formulas} and \ref{e-formulas} we see that $[R_a, L_d] = 0$
and $[ R_a, L_e ] = 0$:
  \allowdisplaybreaks
  \begin{align*}
  -3 [ R_a, L_d]
  &=
  \mathrm{ad}_{[a,d]} + D_{a,d}
  =
  \mathrm{ad}_d + D_{a,d} = 0,
  \\
  -3 [ R_a, L_e]
  &=
  \mathrm{ad}_{[a,e]} + D_{a,e}
  =
  \gamma\mathrm{ad}_e + D_{a,e} = 0.
  \end{align*}
Similarly, using obvious commutation relations (Lemma 3.2 of \cite{BHPU}), we
have
  \allowdisplaybreaks
  \begin{align*}
  [ R_a, D_b ]
  &=
  [ M_b D_b, D_b ] - 3 [ D_b D_c M_d, D_b ]
  =
  [ M_b, D_b ] D_b - 0
  =
  - D_b,
  \\
  [ R_a, D_c ]
  &=
  [ M_c D_c, D_c ] - 3 [ D_b D_c M_d, D_c ]
  =
  [ M_c, D_c ] D_c - 0
  =
  - D_c,
  \\
  [ R_a, M_d ]
  &=
  - [ M_d D_d, M_d ] - 3 [ D_b D_c, M_d ] M_d
  =
  - M_d [ D_d, M_d ]
  =
  - M_d,
  \\
  [ R_a, S ]
  &=
  [ M_a, S ]
  =
  - S.
  \end{align*}
We now apply these formulas to the derivation rule:
  \allowdisplaybreaks
  \begin{align*}
  &
  [ R_a, L_a^s S^t L_b^u D_b^v D_c^w L_c^x M_d^y L_d^z L_e^m]
  =
  L_a^s [R_a, S^t] L_b^u D_b^v D_c^w L_c^x M_d^y L_d^z L_e^m
  \\
  &
  +
  L_a^s S^t [R_a, L_b^u] D_b^v D_c^w L_c^x M_d^y L_d^z L_e^m
  +
  L_a^s S^t  L_b^u [R_a,  D_b^v] D_c^w L_c^x M_d^y L_d^z L_e^m
  \\
  &
  +
  L_a^s S^t L_b^u D_b^v [R_a, D_c^w] L_c^x M_d^y L_d^z L_e^m
  +
  L_a^s  S^t L_b^u D_b^v D_c^w [R_a, L_c^x] M_d^y L_d^z L_e^m
  \\
  &
  +
  L_a^s  S^t L_b^u D_b^v D_c^w L_c^x [R_a, M_d^y] L_d^z L_e^m.
  \end{align*}
Expanding the right side and collecting terms gives the stated result.
\end{proof}

\begin{lemma}\label{abcde-operator}
In $U(\mathbb{M})$ the operator $L_{a^i b^j c^k d^l e^m}$ equals
  \allowdisplaybreaks
  \begin{align*}
  &
  \sum_{\alpha=0}^{\min(j,k)}
  \sum_{\beta=0}^{\alpha}
  \sum_{\kappa=0}^i
  \sum_{\delta=0}^{i-\kappa}
  \sum_{\epsilon=0}^{i-\kappa-\delta}
  (-1)^{i+\alpha-\beta-\kappa-\delta}
  \alpha! \delta! \epsilon!
  \binom{\alpha}{\beta}
  \binom{j}{\alpha,\epsilon}
  \binom{k}{\alpha,\delta}
  \times
  \\
  &\quad
  X_i(\kappa, \delta, \epsilon)
  L_a^{\kappa}
  S^{-\beta-\delta-\epsilon}
  L_b^{j-\alpha-\epsilon}
  D_b^{\delta}
  D_c^{\epsilon}
  L_c^{k-\alpha-\delta}
  M_d^{\alpha+\delta+\epsilon}
  L_d^l
  L_e^m,
  \end{align*}
where $X_i(\kappa, \delta, \epsilon)$ is  polynomial in $\alpha-\beta$
satisfying $X_0(0,0,0) = 1$,
  \[
  X_{i+1}(\kappa,\delta,\epsilon)
  =
  (\alpha{-}\beta{+}\delta{+}\epsilon) X_i(\kappa,\delta,\epsilon)
  +
  X_i(\kappa{-}1,\delta,\epsilon)
  +
  X_i(\kappa,\delta{-}1,\epsilon)
  +
  X_i(\kappa,\delta,\epsilon{-}1),
  \]
and $X_i(\kappa,\delta,\epsilon) = 0$ unless $0 \le \kappa \le i$, $0 \le
\delta \le i{-}\kappa$, $0 \le \epsilon \le i{-}\kappa{-}\delta$.
\end{lemma}

\begin{proof}
Induction on $i$. The basis is Lemma \ref{bcde-operatorforM} and the inductive
step is Lemma \ref{Ra-bracket}. The rest of the proof is a step-by-step
repetition of that of Lemma 4.4 of \cite{BHPU}.
\end{proof}

\begin{lemma}\label{X-function}
We have
  \[
  X_i(\kappa,\delta,\epsilon)
  =
  \binom{\delta{+}\epsilon}{\epsilon}
  \sum_{\zeta=0}^{i{-}\kappa{-}\delta{-}\epsilon}
  \binom{i}{\kappa,\zeta}
  \left\{ \begin{matrix}
  i{-}\kappa{-}\zeta \\ \delta{+}\epsilon
  \end{matrix} \right\}
  (\alpha{-}\beta)^\zeta,
  \]
where the Stirling numbers of the second kind are defined by
  \[
  \left\{ \begin{matrix} r \\ s \end{matrix} \right\}
  =
  \frac{1}{s!} \sum_{t=0}^s (-1)^{s-t} \binom{s}{t} t^r.
  \]
\end{lemma}

\begin{proof}
See \cite{BHPU}, Definition 4.5 and Lemma 4.6.
\end{proof}

The formulas for $L_b$ and $L_c$ in Lemmas \ref{b-formulas} and
\ref{c-formulas} are the same as for the 4-dimensional solvable Malcev algebra
\cite{BHPU}. Thus Lemma 5.1 of \cite{BHPU} holds in our case:
  \allowdisplaybreaks
  \begin{align}
  L_b^u
  &=
  \sum_{\eta=0}^u
  \sum_{\theta=0}^{u-\eta}
  (-1)^{u-\eta-\theta}
  \binom{u}{\eta,\theta}
  S^{u-2\theta}
  M_b^{\eta}
  M_d^{u-\eta}
  D_c^{u-\eta},
  \label{Lb-formula}
  \\
  L_c^x
  &=
  \sum_{\lambda=0}^x
  \sum_{\mu=0}^{x-\lambda}
  (-1)^{x-\lambda}
  \binom{x}{\lambda,\mu}
  S^{x-2\mu}
  M_c^{\lambda}
  M_d^{x-\lambda}
  D_b^{x-\lambda}.
  \label{Lc-formula}
  \end{align}
Therefore Lemma \ref{abcde-operator} can be rewritten in terms of $M_x$, $D_x$
and $S$ as follows.

\begin{lemma}
In $U(\mathbb{M})$ the operator $L_{a^i b^j c^k d^l e^m}$ equals
  \allowdisplaybreaks
  \begin{align*}
  &
  \sum_{\alpha=0}^{\min(j,k)}
  \sum_{\beta=0}^{\alpha}
  \sum_{\kappa=0}^{i}
  \sum_{\delta=0}^{i-\kappa}
  \sum_{\epsilon=0}^{i-\kappa-\delta}
  \sum_{\zeta=0}^{i{-}\kappa{-}\delta{-}\epsilon}
  \sum_{\eta=0}^{j-\alpha-\epsilon}
  \sum_{\theta=0}^{j-\alpha-\epsilon-\eta}
  \sum_{\lambda=0}^{k-\alpha-\delta}
  \sum_{\mu=0}^{k-\alpha-\delta-\lambda}
  \\
  &
  (-1)^{i+j+k+\alpha-\beta-\kappa-\epsilon-\eta-\theta-\lambda}
  (\alpha{-}\beta)^\zeta
  \alpha!
  \binom{\alpha}{\beta}
  (\delta{+}\epsilon)!
  \binom{i}{\kappa,\zeta}
  \times
  \\
  &
  \left\{ \begin{matrix}
  i{-}\kappa{-}\zeta \\ \delta{+}\epsilon
  \end{matrix} \right\}
  \binom{j}{\alpha,\epsilon,\eta,\theta}
  \binom{k}{\alpha,\delta,\lambda,\mu}
  \times
  \\
  &
  M_a^{\kappa}
  S^{j+k-l-2\alpha-\beta-2\delta-2\epsilon-2\theta-2\mu-\gamma m}
  M_b^{\eta} D_b^{k-\alpha-\lambda} D_c^{j-\alpha-\eta}
  M_c^{\lambda} M_d^{j+k+l-\alpha-\eta-\lambda} M_e^m.
  \end{align*}
\end{lemma}

\begin{theorem} \label{structureconstants}
In $U(\mathbb{M})$ the product $( a^i b^j c^k d^l e^m) ( a^r b^n c^p d^q e^s)$
equals
  \allowdisplaybreaks
  \begin{align*}
  &
  \sum_{\alpha=0}^{\min(j,k)}
  \sum_{\beta=0}^{\alpha}
  \sum_{\kappa=0}^{i}
  \sum_{\delta=0}^{i-\kappa}
  \sum_{\epsilon=0}^{i-\kappa-\delta}
  \sum_{\zeta=0}^{i{-}\kappa{-}\delta{-}\epsilon}
  \sum_{\eta=0}^{j-\alpha-\epsilon}
  \sum_{\theta=0}^{j-\alpha-\epsilon-\eta}
  \sum_{\lambda=0}^{k-\alpha-\delta}
  \sum_{\mu=0}^{k-\alpha-\delta-\lambda}
  \sum_{\nu=0}^r
  \\
  &
  (-1)^{i+j+k+\alpha-\beta-\kappa-\epsilon-\eta-\theta-\lambda}
  (\alpha{-}\beta)^\zeta
  \alpha!
  \binom{\alpha}{\beta}
  (\delta{+}\epsilon)!
  \omega^\nu
  \binom{i}{\kappa,\zeta}
  \times
  \\
  &
  \left\{ \begin{matrix}
  i{-}\kappa{-}\zeta \\ \delta{+}\epsilon
  \end{matrix} \right\}
  \binom{j}{\alpha,\epsilon,\eta,\theta}
  \binom{k}{\alpha,\delta,\lambda,\mu}
  \binom{r}{\nu}
  \left[
  \begin{matrix}
  n \\
  k{-}\alpha{-}\lambda
  \end{matrix}
  \right]
  \left[
  \begin{matrix}
  p{+}\lambda \\
  j{-}\alpha{-}\eta
  \end{matrix}
  \right]
  \times
  \\
  &
  a^{r+\kappa-\nu}
  b^{-k+n+\alpha+\eta+\lambda}
  c^{-j+p+\alpha+\eta+\lambda}
  d^{j+k+l+q-\alpha-\eta-\lambda} e^{m+s},
  \end{align*}
where $\omega = j {+} k {-} l {-} 2\alpha {-} \beta {-} 2\delta {-} 2\epsilon
{-} 2\theta {-} 2\mu {-} m\gamma$. (For $(\alpha{-}\beta)^\zeta$ we set $0^0 =
1$.)
\end{theorem}

\begin{proof}
Apply $L_{a^i b^j c^k d^l e^m}$ to $a^r b^n c^p d^q e^s$.
\end{proof}


\section{ The universal alternative enveloping algebra}

By Lemma 6.3 of \cite{BHPU} we have
  \[
  (c, ab, ab)= -bd,
  \quad
  (b, ac, ac) = cd,
  \quad
  (a, bc, bc) = 2d^2.
  \]
Using Theorem \ref{structureconstants} (of the present paper) we calculate
  \[
  ( ac+be, ac+be, a ) = ( ac, be, a ) + ( be, ac, a ) = de.
  \]
Let $J$ be the ideal of $U(\mathbb{M})$ generated by $\{ bd, cd, d^2, de \}$.
We consider $U(\mathbb{M})/J$, and our goal is to prove that this is the
universal alternative enveloping algebra of $\mathbb{M}$. Since (the cosets of
) the elements $d^2$, $cd$, $bd$, $de$ are zero in $U(\mathbb{M})/J$ we can
reduce each basis monomial of $U(\mathbb{M})$ modulo $J$ to either $a^i d$
(type 1) or $a^r b^n c^p e^s$ (type 2).

\begin{lemma} \label{multiplicationtable}
In $U(\mathbb{M})/J$ we have
  \allowdisplaybreaks
  \begin{align*}
  a^i d \cdot a^r d
  &=
  0,
  \\
  a^i d \cdot a^r b^n c^p e^s
  &=
  \delta_{0n} \delta_{0p} \delta_{0s} a^i (a{-}1)^r d,
  \\
  a^i b^j c^k e^m \cdot a^r d
  &=
  \delta_{j0} \delta_{k0} \delta_{m0} a^{i+r} d,
  \\
  a^i b^j c^k e^m \cdot a^r b^n c^p e^s
  &=
  a^i (a{+}j{+}k{-}\gamma m)^r b^{j+n} c^{k+p} e^{m+s}
  +
  \delta_{m,0} \delta_{s,0} \delta_{j+n, 1}\delta_{k+p,1} T^{i r}_{jk},
  \end{align*}
where
  \[
  T^{ir}_{jk} =
  \begin{cases}
  \;
  0
  &\text{if $(j,k)=(0,0)$},
  \\
  \;
  (a-1)^{i+r} d - a^i (a+1)^r d
  &\text{if $(j,k)=(1,0)$},
  \\
  \;
  - (a-1)^{i+r} d - a^i (a+1)^r d
  &\text{if $(j,k)=(0,1)$},
  \\
  \;
  a^i (a-1)^r d - a^i (a+2)^r d
  &\text{if $(j,k)= (1,1)$}.
  \end{cases}
  \]
\end{lemma}

\begin{proof}
The proof of the first equation, and of the last in the case $s = m = 0$, is
given in Lemma 6.5 of \cite{BHPU}. For the second equation, we compute $L_{a^i
d}$ using Lemma \ref{abcde-operator}. Since $j=k=0$ and $m=0$ we have that
$\alpha=\beta=0$. Therefore $L_{a^i d}$ equals
  \[
  \sum_{\kappa=0}^i
  \sum_{\delta=0}^{i-\kappa}
  \sum_{\epsilon=0}^{i-\kappa-\delta}
  (-1)^{i-\kappa-\delta}
  \delta!
  \epsilon!
  \binom{0}{0,\varepsilon}
  \binom{0}{0,\delta}
  X_i(\kappa,\delta,\epsilon)
  L_a^{\kappa}
  S^{-\delta-\epsilon}
  L_b^{-\epsilon}
  D_b^{\delta}
  D_c^{\epsilon}
  L_c^{-\delta}
  M_d^{\delta+\epsilon}
  L_d.
  \]
Clearly $\epsilon = \delta = 0$, and since $X_i(\kappa, 0, 0) = 0$ unless
$\kappa = i$ by Lemma \ref{X-function}, we get $L_{a^i d} = L^i_a L_d$.
Therefore in $U(\mathbb{M})/J$ we have
  \[
  L_{a^i d} ( a^r b^n c^p e^s )
  =
  M_a^i S^{-1} M_d ( a^r b^n c^p e^s )
  =
  \delta_{s0} \delta_{n0} \delta_{p0} a^i (a{-}1)^r d.
  \]
For the third equation, we consider $L_{a^i b^j c^k e^m}$. Since the exponent
of $d$ is $\le 1$ in monomials of both types 1 and 2, we have $\alpha + \delta
+ \epsilon = 0$ in Lemma \ref{abcde-operator}. Hence $\alpha = \delta =
\epsilon = 0$ and so $\beta = 0$. Therefore,
  \allowdisplaybreaks
  \begin{align*}
  L_{a^ib^jc^ke^m}
  =
  \sum_{\kappa=0}^i
  (-1)^{i-\kappa}
  X_i(\kappa, 0, 0)
  L_a^{\kappa} S^0 L_b^j L_c^k= L_a^i L_b^j L_c^k.
  \end{align*}
Using equations \eqref{Lb-formula} and \eqref{Lc-formula} we obtain the stated
result.  It remains only to consider the fourth equation in the case $s + m \ne
0$. Since the exponent of $e$ in Theorem \ref{structureconstants} is non-zero,
the exponent of $d$ must be zero. Hence $j + k - \alpha - \eta - \lambda = 0$
and so $\alpha + \eta + \lambda = j + k$. But then, since $\alpha + \eta \le
j$, $\lambda \le k$, $\eta \le j$, $\alpha + \lambda \le k$, we see that
$\lambda = k$, $\eta = j$, $\alpha = \beta = 0$, $\delta = \epsilon = \theta =
\mu = \zeta = 0$. The sum collapses to
  \[
  \sum_{\nu=0}^r
  (j{+}k{-}\gamma m)^{\nu}
  \binom{r}{\nu}
  a^{i+r-\nu}
  b^{j+n}
  c^{k+p}
  e^{m+s}
  =
  a^i (a{+}j{+}k{-}\gamma m)^r b^{j+n} c^{k+p} e^{m+s}.
  \]
The proof is complete.
\end{proof}

\begin{theorem}
The universal alternative enveloping algebra $A(\mathbb{M})$ has the basis $\{
\, a^i d, \, a^i b^j c^k e^m \mid i,j,k \ge 0 \, \}$ and the structure
constants of Lemma \ref{multiplicationtable}.
\end{theorem}

\begin{proof}
We compute all possible associators of basis monomials of $U(\mathbb{M})/J$ of
types 1 and 2. Since the product of a monomial of type 1 with any monomial is a
linear combination of monomials of type 1, and the product of two monomials of
type 1 is zero, it is clear that every associator with two monomials of type 1
vanishes. Next consider an associator with one monomial of type 1 and two of
type 2, for example
  \[
  ( a^i d, a^r b^n c^p e^s, a^l b^k c^t e^q )
  =
  ( a^i d \cdot a^r b^n c^p e^s ) a^l b^k c^t e^q
  -
  a^i d ( a^r b^n c^p e^s \cdot a^l b^k c^t e^q ).
  \]
If $s + q \ne 0$, then a straightforward application of Lemma
\ref{multiplicationtable} shows that the result is zero.  If $s + q = 0$, then
Theorem 6.6 of \cite{BHPU} shows that $(a^i d, a^r b^n c^p, a^l b^k c^t) = 0$.
The other two cases are similar. We finally consider three monomials of type 2:
  \[
  ( a^i b^j c^k e^m, a^r b^n c^p e^s, a^l b^q c^t e^y )
  =
  A - B + C - D,
  \]
where
  \allowdisplaybreaks
  \begin{align*}
  A - B
  &=
  a^i (a{+}j{+}k{-}\gamma m)^r b^{j+n} c^{k+p} e^{m+s} \cdot a^l b^q c^t e^y
  \\
  &\quad
  -
  a^i b^j c^k e^m \cdot a^r (a{+}n{+}p{-}\gamma s)^l b^{n+q} c^{p+t} e^{s+y}
  \\
  &=
  \delta_{j+n+q, 1} \delta_{k+p+t, 1}
  \left(
  \begin{array}{l}
  \sum_{\nu=0}^r
  \binom{r}{\nu}
  (j{+}k{-}\gamma m)^{\nu}
  \delta_{m+s,0}
  \delta_{y,0}
  T_{j+n,p+k}^{i+r-\nu,l}
  \\[6pt]
  \quad
  -
  \sum_{\mu=0}^l
  \binom{l}{\mu}
  (n{+}p{-}\gamma s)^{\mu}
  \delta_{m,0}
  \delta_{s+y,0}
  T_{jk}^{i, r+l-\mu})
  \end{array}
  \right),
  \\[6pt]
  C - D
  &=
  \delta_{m,0}
  \delta_{s,0}
  \delta_{j+n,1}
  \delta_{k+p,1}
  T_{jk}^{ir}
  a^l b^q c^t e^y
  -
  \delta_{s0}
  \delta_{y0}
  \delta_{n+q,1}
  \delta_{p+t,1}
  a^i b^j c^k e^m
  T_{np}^{rl}.
\end{align*}
If $y \ne 0$ then $s + y \ne 0$ and so $A - B = 0$; moreover, the second term
in $C - D$ vanishes, and the second equation of Lemma \ref{multiplicationtable}
shows that the first term in $C - D$ also vanishes, so $C - D = 0$. Similar
arguments apply if $m \ne 0$ or $s \ne 0$. Finally, if $s + m + y = 0$, then we
know the value of the corresponding associator from Theorem 6.6 of \cite{BHPU},
and hence the alternativity property is clear.
\end{proof}

\begin{corollary}
Every Malcev algebra $\mathbb{M}_\gamma$ in the one-parameter family of
solvable 5-dimensional Malcev algebras is special; that is, $\mathbb{M}_\gamma$
is isomorphic to a subalgebra of the commutator algebra of an alternative
algebra.
\end{corollary}

We can construct much smaller (but still infinite-dimensional) alternative
enveloping algebras for the Malcev algebras $\mathbb{M}_\gamma$ as follows. We
consider the algebra $\mathbb{A}_\gamma$ with basis $\{ \, a^r, b, c, d, e \mid
r \in \mathbb{Z}, \, r \ge 1 \, \} $; the nonzero products of basis elements
are $a^r \cdot a^s = a^{r+s}$ together with
  \begin{equation} \label{newalgebra}
  \;
  a^r \cdot d = d,
  \quad
  b \cdot a^r = b,
  \quad
  b \cdot c = d,
  \quad
  c \cdot a^r = c,
  \quad
  c \cdot b = -d,
  \quad
  e \cdot a^r = (-\gamma)^r e.
  \end{equation}
It is easy to check that $\mathbb{M}_\gamma$ is isomorphic to the subalgebra of
$\mathbb{A}_\gamma^-$ with basis $\{ a, b, c, d, e \}$. It remains to verify
that $\mathbb{A}_\gamma$ is alternative; equivalently, that the associator
$(x,y,z)$ is a skew-symmetric function of $x, y, z \in \mathbb{A}_\gamma$.
Since the associator is trilinear it suffices to check basis elements. We write
$\mathbb{A}_\gamma' = \mathrm{span}\{ a^r, b, c, d \}$ and $\mathbb{I}_\gamma =
\mathrm{span}\{ a^t - a^s \}$ for all $s,t>0$. It follows from
\eqref{newalgebra} that $\mathbb{A}_\gamma'$ is a subalgebra of
$\mathbb{A}_\gamma$, that $\mathbb{I}_\gamma$ is an ideal in
$\mathbb{A}_\gamma'$, and that $\mathbb{A}_\gamma'/\mathbb{I}_\gamma$ is
isomorphic to the 4-dimensional alternative algebra of \cite[Table 3]{BHPU}.
Consider arbitrary elements $x + \mathbb{I}_\gamma$, $y + \mathbb{I}_\gamma$,
$z + \mathbb{I}_\gamma$ in $\mathbb{A}_\gamma'/\mathbb{I}_\gamma$. Since the
quotient algebra is alternative, we have
  \[
  (x,y,z) + \mathbb{I}_\gamma
  =
  (x + \mathbb{I}_\gamma, y + \mathbb{I}_\gamma, z + \mathbb{I}_\gamma)
  =
  (y + \mathbb{I}_\gamma, x + \mathbb{I}_\gamma, z + \mathbb{I}_\gamma)
  =
  (y,x,z) + \mathbb{I}_\gamma.
  \]
Hence $(x,y,z) - (y,x,z) \in \mathbb{I}_\gamma$. But \eqref{newalgebra} implies
that associators take values in $\mathbb{L} = \mathrm{span}\{b,c,d,e\}$.  Since
$\mathbb{I}_\gamma \cap \mathbb{L} = \{0\}$, we obtain $(x,y,z) = (y,x,z)$. A
similar argument shows that $(x,y,z) = (x,z,y)$, and so $\mathbb{A}_\gamma'$ is
alternative.

It remains to consider the cases in which at least one of $x, y, z$ is $e$.
First, we show that $(e,x,y) = (x,e,y) = (x,y,e) = 0$ for any $x, y \in
\mathbb{L}$: for this, it is easy to see that $\mathbb{L}^2 =
\mathrm{span}\{d\}$ and $\mathbb{L}^3 = \{0\}$. Second, suppose that $x = a^s$
and $y \in \mathrm{span}\{b,c,d\}$, then $(e, a^s, y)= (e \cdot a^s) y - e (a^s
\cdot y) = 0$ by \eqref{newalgebra}; the other five permutations $( e, a^s, y
)$ are similar. Finally, for the remaining cases direct calculation easily
shows that
  \[
 (e, a^r, e) = (a^r, e, e) = (e, e, a^r) =
 (e, a^r, a^s) = (a^r, e, a^s) = (a^r, a^s, e) = 0.
  \]
Thus $\mathbb{A}_\gamma$ is alternative.


\section{Conclusion}

It follows from results of Elduque \cite{Elduque} that, unlike the
4-dimensional solvable Malcev algebra considered in \cite{BHPU}, the Malcev
algebras $\mathbb{M} = \mathbb{M}_\gamma$ in the one-parameter family of
5-dimensional solvable algebras are not isomorphic to subalgebras of the
7-dimensional simple Malcev algebra. This raises the question of the existence
of a finite-dimensional alternative algebra $\mathbb{A}_\gamma$ for which the
commutator algebra $\mathbb{A}_\gamma^-$ contains a subalgebra isomorphic to
$\mathbb{M}_\gamma$.  If such an algebra $\mathbb{A}_\gamma$ does not exist for
some $\gamma$, then we have an example of a finite-dimensional special Malcev
algebra which does not have a finite-dimensional enveloping alternative
algebra.  A related result is the generalization of the Ado theorem in
\cite{PerezIzquierdoShestakov}; this guarantees the existence of a
finite-dimensional nonassociative enveloping algebra for any finite-dimensional
Malcev algebra, but this enveloping algebra is not necessarily alternative.
This leads to an important open problem: If $\mathbb{M}$ is a special
finite-dimensional Malcev algebra, then does $\mathbb{M}$ necessarily have a
finite-dimensional alternative envelope?


\end{document}